\documentclass[a4paper,11pt]{article}         
\usepackage{amsmath,amsfonts,amsthm,amssymb}  
\usepackage[utf8]{inputenc} 
\usepackage[T1]{fontenc}    
									   
\DeclareMathOperator*{\argmin}{arg\,min} 

\usepackage{graphicx,caption}            
\usepackage{color}                       
\usepackage{hyperref}                    
\usepackage{grffile}

\usepackage{makeidx}                  
\usepackage{acronym}
\usepackage{mathrsfs}
\usepackage{float}
\usepackage{vruler}                    
\usepackage{palatino, url, multicol}
\usepackage{setspace}                  
\usepackage{txfonts}

\usepackage{tikz}
\usepackage{tikz-3dplot}
\usetikzlibrary{shapes.geometric, arrows, calc, 3d}

\usepackage[numbers,sort&compress]{natbib}   
\usepackage{import}

\usepackage{caption}

\usepackage{listings}

\usepackage{comment}

\newtheorem{theorem}{Theorem}[section]

\newtheorem{lemma}[theorem]{Lemma}

\restylefloat{figure}

\begin{document}

\title{Form Convex Hull to Concavity: \\ Surface Contraction Around a Point Set}
\author{\Large{Netzer Moriya}}

\date{}

\maketitle

\begin{abstract}
This paper investigates the transformation of a convex hull, derived from a d-dimensional point cloud, into a concave 
surface (\( S^{cc} \)). Our primary focus is on the development of a methodology that ensures all points in the point cloud 
are encapsulated within a closed, non-intersecting concave surface. 
The study begins with the initial convex hull (\( S^{ch} \)) and employs an iterative process of facet replacement and expansion 
to evolve the surface into \( S^{cc} \), which accurately conforms to the complex geometry of the point cloud.

Central to our analysis is the formulation and proof of the \textit{Point Cloud Contraction Theorem} – which provides a 
theoretical foundation for the transformation process. This theorem rigorously establishes that the iterative approach will 
result in a surface that encompasses all points in the point cloud, thereby lending mathematical rigor to our methodology.

Through empirical analysis, we demonstrate the effectiveness of this process in various scenarios, highlighting the 
method's capability to include all points of the cloud while preserving topological integrity. The research additionally explores 
the computational complexity of the transformation process, revealing a range between \( O(n^2) \) in complex cases 
to \( O((n-m) \log (n-m)) \) in more typical scenarios. These findings have practical implications in areas such as 3D reconstruction, 
computational geometry, or spatial data analysis, where precise modeling of both convex and concave structures is crucial.

\end{abstract}

\section{Introduction}
We investigate the transformation of a convex hull, derived from a d-dimensional point cloud, into a more comprehensive 
concave surface, denoted as \( S^{cc} \). This surface is unique in that it encompasses all points \( p_i \) of the point 
cloud \( P \), and is constrained to be closed. We explore the dynamics of this process and examine its relationship to the 
concept of concavity in geometric terms..

The convex hull, represented as \( S^{ch} \), is the smallest convex set that encloses all points in a given d-dimensional space 
\cite{Loffler:2010, Moriya:2023smallest}. 
The modification of \( S^{ch} \) to include additional points from \( P \) leads to a transformation resulting in the formation of 
concave regions \cite{Karzanov:2005}. This change is characterized by the emergence of inward curves or hollows, 
where the internal angles exceed 180 degrees. 
The transition from \( S^{ch} \) to \( S^{cc} \) is iterative, continuing until all points from \( P \) are incorporated onto the 
surface of \( S^{cc} \).

\paragraph{}
We highlight the following geometric implications arising from the transformation of \( S^{ch} \) to \( S^{cc} \):

\begin{itemize}
  \item The construction of \( S^{cc} \) underscores the intrinsically non-convex characteristics of the original point 
  cloud \( P \), revealing its complex geometry.
  \item This transformation is critical for accurately depicting the true geometry of the point set, particularly in the presence 
  of marked concavities.
  \item A significant challenge in developing \( S^{cc} \) is ensuring that the resulting facets do not intersect, thereby 
  maintaining the integrity and authentic representation of the point cloud's structure.
  \item The principles involved in this process are of paramount importance in domains such as 3D reconstruction and computational 
  geometry, where precise geometric modeling is essential.
\end{itemize}

Our research endeavors to provide a comprehensive understanding of the process involved in transforming the convex hull \( S^{ch} \) 
into a concave surface \( S^{cc} \), which includes ensuring the surface remains closed throughout the transformation. 
This procedure is pivotal for fully representing a d-dimensional point cloud \( P \) and underscores the fundamental concept of 
concavity in geometric analysis.

\paragraph{}
In the context of transitioning from \( S^{ch} \) (the convex hull) to \( S^{cc} \) (the concave surface), various related 
methodologies can be employed. These approaches aim to evolve the convex hull into a more comprehensive surface that accurately 
represents the underlying structure of the point cloud. Some of these methods include:

\begin{itemize}
  \item \textbf{Alpha Shape} is a concept from computational geometry that generalizes the idea of a convex hull. 
  An alpha shape \cite{Asaeedi:2013} is 
  defined for a particular value of alpha, which determines the level of detail in the shape's boundary. 
  \begin{itemize}
    \item For a large value of alpha, the alpha shape may resemble the convex hull.
    \item As alpha decreases, the alpha shape adapts to capture more of the concavities in the data, potentially resembling the 
    process outlined in our paper where more points from the point cloud are incrementally included in the surface.
    \item With small alpha values, the alpha shape may lead to the formation of holes in the surface.
    \item The relevance to our study lies in the alpha shape's ability to represent varying degrees of concavity, making it a 
    potential model for understanding the transformation from convex to concave hulls.
  \end{itemize}

  Although Alpha Shapes serve as a potent tool for approximating the shape of a point cloud, they are not inherently crafted to 
  ensure a final surface akin to \( S^{cc} \). Specifically, Alpha Shapes may not guarantee a surface that includes all points 
  in a non-intersecting and closed manner, as is the characteristic requirement of \( S^{cc} \).

  \item \textbf{Concave Hull} is a concept that extends beyond the limits of the convex hull to include indentations that follow the 
  shape of the data more closely. 
  \begin{itemize}
    \item Unlike the convex hull, a concave hull \cite{Park:2012} can wrap around concave areas and holes in the data, much like the 
    surface \( S^{cc} \) in our paper.
    \item The concave hull is particularly relevant when the point cloud exhibits significant concavities, which the convex 
    hull fails to represent.
    \item The transformation to a concave hull can be seen as a specific instance of the process we discuss, where the surface 
    evolves to include more points from the point cloud.
  \end{itemize}
  
  The concave hull method, while designed to more accurately fit data points than a convex hull by capturing concavities, does not 
  intrinsically meet the specific objectives of our study. These objectives include ensuring that all points of the point cloud are 
  incorporated on the surface and that the surface remains non-intersecting. Additionally, adapting a concave hull algorithm to 
  construct a closed surface for an arbitrary (3D) point cloud presents challenges. It is not a straightforward task and may 
  necessitate substantial customizations to the conventional methodology.

  \item \textbf{Minimum Volume Enclosing Surface} is a concept aimed at finding the surface that encloses a set of points with the 
  minimal possible volume \cite{Moriya:2024TheLargest}. 
  \begin{itemize}
    \item This concept is relevant when the goal is to closely fit a surface around a set of points \cite{Pramila:1986}, similar 
	to the aim of our dynamic transformation process.
    \item In the context of our paper, the transformation from a convex to a concave hull can be seen as an attempt to minimize the 
    'unused' volume within the hull, albeit with considerations for the shape's concavities and complexities.
  \end{itemize}
  
  Both the Minimum Volume Enclosing Surface (MVES) method and the transformation from \( S^{ch} \) to \( S^{cc} \) aim to closely 
  fit a set of points. However, their methodologies, objectives, and the assurances they provide differ markedly. 
  The MVES method, while intended to generate a closed surface, does not inherently guarantee that the surface will include all 
  points of an arbitrary 3D point cloud, particularly in a manner that adheres to the intricate geometry of the point cloud, as 
  is characteristic of \( S^{cc} \). Additionally, ensuring that the resultant surface from the MVES method is free of 
  self-intersections or overlaps presents an extra layer of complexity.
  
\end{itemize}
  
The transformation from \( S^{ch} \) to \( S^{cc} \) highlights the limitations of convex hulls in accurately representing the 
geometric intricacies of a point cloud. 
The development of \( S^{cc} \) is particularly relevant in the context of surface reconstruction, mesh generation, and shape 
analysis, where the need to capture both the global shape and local details of a point set is paramount. Additionally, our 
approach may find applications in fields such as computer graphics, where realistic modeling of objects often requires a 
balance between convex and concave features, and in spatial data analysis, where the accurate representation of geographical 
features is crucial \cite{Yang:2010}.  
The concept of transforming a convex hull into a more detailed concave surface \cite{Alliez:2007} aligns with the broader goals 
in these fields to 
develop algorithms and methodologies that are more attuned to the underlying spatial structures of data. 
 
In the subsequent sections, we detail the mathematical methodology employed to evolve \( S^{ch} \) into \( S^{cc} \). 
This includes a conceptual proof of a theorem which guarantees the inclusion of all points \( p_i \) in the point cloud \( P \) 
on the surface \( S^{cc} \) (the \textit{Point Cloud Contraction Theorem}). 
Our approach adeptly captures concavities and intricate features, facilitating this transition and enhancing the accuracy and 
depth of the models.

Additionally, we present simulated results that illustrate the efficacy of this approach in 3D environments.

\section{Problem Definition}
We describe a theoretical surface dynamically conforming to a set of fixed points in d-dimensional space, culminating in a 
state where the surface intersects every point in the set.

Consider a set \( P = \{p_1, p_2, \ldots, p_n\} \subset \mathbb{R}^d \), representing a collection of fixed points in 
d-dimensional space. Define a surface \( S \) that initially encapsulates \( P \) in its convex hull, \( \text{conv}(P) \). 
The evolution of \( S \) is characterized by a continuous contraction process under the following conditions:

\begin{enumerate}
    \item \textbf{Initial Configuration of \( S \)}: The surface \( S \) initially coincides with the convex hull of \( P \), given by 
    \[ \text{conv}(P) = \left\{ \sum_{i=1}^{n} \lambda_i p_i \, \middle| \, p_i \in P, \lambda_i \geq 0, \sum_{i=1}^{n} \lambda_i = 1 \right\}. \]
	
    \item \textbf{Contraction Dynamics of \( S \)}:
    \begin{itemize}
        \item The surface \( S \) undergoes a continuous contraction process, denoted by \( \{S_k\}_{k=0}^\infty \), where \( S_k \) represents the state of \( S \) at iteration \( k \).
        \item This contraction process is defined as follows:
        \begin{align*}
            S_0 &= \text{conv}(P) \\
            S_k &\rightarrow S_{k+1} \quad \text{as} \quad k \rightarrow \infty
        \end{align*}
        where \( \text{conv}(P) \) is the convex hull of the point set \( P \), and \( S_k \) evolves to approach the spatial arrangement of points in \( P \).
        \item The contraction process is unrestricted, allowing \( S \) to deform continuously and take on any shape necessary to closely approach the points in \( P \).
    \end{itemize}

    \item \textbf{Facet Formation}:
    \begin{itemize}
        \item During the contraction process of \( S \), it gives rise to a set of flat facets \( F_k \) at iteration \( k \) 
		during the contraction process. Each facet is represented as a hyperplane defined by a subset of points \( C \subseteq P \), 
		where \( |C| \geq d+1 \).
        \item The geometry of these facets is dictated by the spatial arrangement of the points within \( C \). Specifically, the 
    	facet hyperplane is determined by finding the hyperplane that best fits the points in \( C \), minimizing some suitable 
    	measure of deviation from those points.
        \item These facets serve as a discretization of the surface \( S \) as it undergoes the contraction process. They provide 
    	a way to break down the continuous surface into flat regions defined by subsets of points, aiding in the analysis of the 
    	evolving surface.
    	\item As a fundamental property of the contraction process, we assume that 
		\( \forall i, j \neq i \in \mathbb{N}, \quad (F_i \cap F_j) = \emptyset \), i.e., the facets formed during the process 
		remain non-intersecting and asserting that for all natural numbers \( i \) and \( j \) where \( i \neq j \), the 
		intersection of facets \( F_i \) and \( F_j \) is an empty set, denoted by \( \emptyset \), ensuring non-intersecting 
		facets and a well-defined surface.
    \end{itemize}

    \item \textbf{Termination of Contraction}:
    \begin{itemize}
        \item The contraction terminates when the internal region bounded by \( S \) reaches minimal volume extent consistent 
		with the positions of points in \( P \).
        \item At this stage, the surface \( S \) intersects every point in \( P \), such that \( \forall p_i \in P, p_i \in S \).
    \end{itemize}

    \item \textbf{Final Structure of \( S \)}:
    \begin{itemize}
        \item The final state of \( S \) is a geometric surface that conforms precisely to the point set \( P \), representing 
		a detailed boundary that intersects every point in \( P \).
        \item The final state of \(S\) is not only a geometric surface that precisely conforms to the point set \(P\) but also 
		represents an optimal solution in terms of minimizing its areal surface while ensuring it intersects every point in \(P\)
		ensuring that the surface maintains non-cross-intersecting facets, that is:	

		\begin{align*}
		S = \argmin_{S'} \Bigg\{ &\text{Area}(S') \; \Bigg| \; \forall p_i \in P, \exists x_i \in S' \text{ closest to } p_i, \\
		&\text{and } \forall F_i, F_j \in \mathcal{F}(S'), \, i \neq j \Rightarrow F_i \cap F_j = \emptyset \\
		&\text{or } \partial F_i \cap \partial F_j \Bigg\}.
		\end{align*}

        Here, \(\text{Area}(S')\) represents the surface area of \( S' \), \(\mathcal{F}(S')\) denotes the set of facets forming the 
		surface \( S' \), and \(\partial F_i\) is the boundary of facet \( F_i \). This expression encapsulates the surface's 
		conformance to \( P \), minimal area criterion, and non-cross-intersection constraint among its facets.
    \end{itemize}
\end{enumerate}

\section{Mathematical Approach}
In the following, we aim to prove that given a set \( P = \{p_1, p_2, \ldots, p_n\} \subset \mathbb{R}^d \), 
the contraction process \( \{S_k\}_{k=0}^\infty \) ends when $\forall P_i \in P \quad d(S_{k+1}, P) < d(S_k, P)$.

\vspace{0.4cm}

\begin{theorem}[Point Cloud Contraction Theorem]\label{thm:Contraction to Point Cloud Envelope}

Let \( P = \{p_1, p_2, \ldots, p_n\} \) be a finite set of distinct points in d-dimensional space \( \mathbb{R}^d \), 
and let \( S \) be a surface that initially coincides with the convex hull of \( P \), denoted as \( \text{conv}(P) \). 
Suppose \( S \) undergoes a continuous contraction process defined by the following properties:

\begin{enumerate}
    \item \textbf{Initial Configuration}: \( S \) initially is the convex hull of \( P \), where 
    \[ \text{conv}(P) = \left\{ \sum_{i=1}^{n} \lambda_i p_i \, \middle| \, p_i \in P, \lambda_i \geq 0, \sum_{i=1}^{n} \lambda_i = 1 \right\}. \]

    \item \textbf{Contraction Dynamics}: During the contraction process, \( S \) progressively adapts to the spatial arrangement of 
	points in \( P \), without any restrictions on the form it can assume, and it may form flat facets determined by subsets of 
	points \( C \subseteq P \). That is:
	\[ S_k \xrightarrow{\text{adapt}} S_{k+1} = \bigcup_{C \subseteq P, |C| \geq d} \text{facet}(C), \quad \forall k \in \mathbb{N}, \]

    \item \textbf{Termination Condition}: The contraction terminates when the internal region bounded by \( S \) reaches a minimal 
	spatial extent consistent with the positions of points in \( P \), resulting in \( S \) intersecting every point in \( P \). 
	In other words:
    The contraction terminates when the following condition is met:
    \[
       \min_{S'} \left\{ \text{Volume}(\text{Int}(S')) \; \middle| \; P \subseteq \text{Int}(S') \subseteq \mathbb{R}^d \right\},
    \]
    such that \(\forall p_i \in P\), \(p_i \in S\).
\end{enumerate}

Then, the final structure of \( S \) after the completion of the contraction process is such that every point in \( P \) 
lies on \( S \), i.e., \( \forall p_i \in P, p_i \in S \).

\end{theorem}

\vspace{0.2cm}

\begin{proof}\ref{thm:Contraction to Point Cloud Envelope}

Lemma \ref{lmm:The convex hull of P, encapsulates all points in P} ensures that all $p_i$ are initialy enclosed in $S$. 
Lemma \ref{lmm:contraction process continues until} prooves that \( \forall P_i \in P, \quad P_i \in S_m \).

\vspace{0.2cm}

Hence, following \( \{S_k\}_{k=0}^\infty \), every point in \( P \) lies on \( S_m \), i.e., \( \forall p_i \in P, p_i \in S_m \).

\end{proof}

\section{Proofs for Lemmas}

\begin{lemma}\label{lmm:The convex hull of P, encapsulates all points in P}
Let \( P = \{p_1, p_2, \ldots, p_n\} \) be a finite set of distinct points in \( \mathbb{R}^d \), and let \( S \) be the convex hull of \( P \). Then, \( S \) encapsulates all points in \( P \).
\end{lemma}

\begin{proof}\ref{lmm:The convex hull of P, encapsulates all points in P}
Consider \( S \) defined as the convex hull of \( P \), denoted \( \text{conv}(P) \). By definition, \( \text{conv}(P) \) is the set comprising all convex combinations of points in \( P \), i.e., 
\[ \text{conv}(P) = \left\{ \sum_{i=1}^{n} \lambda_i p_i \, \middle| \, p_i \in P, \lambda_i \geq 0, \sum_{i=1}^{n} \lambda_i = 1 \right\}. \]

To establish that \( S \) encapsulates every point in \( P \), it suffices to demonstrate that for any \( p_j \in P \), \( p_j \) is an element of \( S \). Consider any \( p_j \in P \). The point \( p_j \) can be expressed as a convex combination of points in \( P \) where \(\lambda_j = 1\) and \(\lambda_i = 0\) for all \( i \neq j \). This particular combination is trivially a member of \( \text{conv}(P) \), thus ensuring that \( p_j \in S \).

Furthermore, the set \( S \) being a convex hull, is convex by nature. This implies that for any two points \( x, y \in S \), the line segment joining \( x \) and \( y \) lies entirely within \( S \). Therefore, \( S \) not only contains the points of \( P \) but also all line segments joining these points, reinforcing the encapsulation of \( P \) within \( S \).

\vspace{0.2cm}

Hence, we conclude that the convex hull \( S \) of the set \( P \) encapsulates all points in \( P \).

\end{proof}

\begin{lemma}\label{lmm:contraction process continues until}

   Given a set \( P = \{P_1, P_2, \ldots, P_n\} \subseteq \mathbb{R}^d \) and a surface \( S \), consider a contraction 
   process \( \{S_k\}_{k=0}^\infty \) such that:
   \begin{itemize}
       \item For each \( k \), \( S_k \rightarrow S_{k+1} \) implies \( d(S_{k+1}, P) < d(S_k, P) \), where \( d \) denotes a 
	   suitable metric.
       \item Facets are formed for \( C \subseteq P \), \( S|_C \) planar, aligning with \( C \).
       \item The process terminates at \( m \) if \( \forall P_i \in P, P_i \in S_m \) and \( S_m \) is minimal.
       \item If \( \exists P_i \in P \) with \( P_i \notin S_m \), extend to \( m' \) where \( P_i \in S_{m'} \).
   \end{itemize}

\end{lemma}

\begin{proof}\ref{lmm:contraction process continues until}

\begin{enumerate}

\item \textit{Contraction Process and Metric Definition}: 
    Define the distance metric \( d(S, P) \) as a sum of the squared Euclidean distances from each point in \( P \) to the nearest 
	point on the surface \( S \):
    \[ d(S, P) = \sum_{p_i \in P} \min_{x \in S} \|p_i - x\|^2. \]
    For each iteration \( k \), as \( S_k \) transitions to \( S_{k+1} \), we ensure that:
    \[ d(S_{k+1}, P) < d(S_k, P). \]
    This implies a monotonically decreasing sequence of distances, indicating progressive convergence of the surface towards the 
	points in \( P \).

\item \textit{Convergence Analysis:} 
    To prove convergence, we need to show that as \( k \rightarrow \infty \), \( d(S_k, P) \rightarrow 0 \), implying that every 
	point in \( P \) lies on \( S \) in the limit.
    Since \( d(S_{k+1}, P) < d(S_k, P) \) and \( d(S, P) \geq 0 \) for all \( S \) and \( P \), we have a bounded, monotonically 
	decreasing sequence. By the Monotone Convergence Theorem \cite{rudin:1976, Ugwunnadi:2019}, this sequence converges.
    Let \( \epsilon > 0 \) be given. Since \( \{d(S_k, P)\} \) is convergent, there exists an \( N \) such that for 
	all \( k \geq N \), \( |d(S_k, P) - L| < \epsilon \), where \( L \) is the limit of \( \{d(S_k, P)\} \). 
	In our case, \( L = 0 \) as the surface \( S \) must intersect all points in \( P \) in the limit.

\item \textit{Topological Closure and Hausdorff Distance:} 
    Consider the topological closure of \( S \), denoted \( \text{Cl}(S) \), and the Hausdorff distance \( d_H \). The process 
	terminates at \( m \) when:
    \[ \forall p_i \in P, p_i \in \text{Cl}(S_m) \quad \text{and} \quad d_H(S_m, P) \leq d_H(S_k, P) \, \forall k < m. \]
    The Hausdorff distance \( d_H \) between \( S \) and \( P \) will also approach 0 as \( k \rightarrow \infty \), reinforcing 
	the convergence of \( S \) to a state where it intersects all points in \( P \).

\item \textit{Facet Formation and Non-Intersecting Surface}:
	Let \( F_{k+1} \) be a newly formed facet during the transition from \( S_k \) to \( S_{k+1} \). Define \( \mathcal{H}(C) \) as 
	the hyperplane formed by \( C \subseteq P \). The facet formation follows:
	\[ F_{k+1} = \mathcal{H}(C) \quad \text{with} \quad C \subseteq P, |C| \geq d+1. \]
	To ensure non-intersection, we impose:
	\[ \forall F_i \in S_k, F_i \cap F_{k+1} \subseteq \mathcal{H}(C \cap C_i) \]
	where \( C_i \) are the points defining \( F_i \). This condition ensures that any intersection between \( F_i \) and \( F_{k+1} \) 
	occurs only along shared vertices or edges, preventing interior intersections and preserving the integrity of \( S \).
\end{enumerate}

\vspace{0.2cm}

Thus, the contraction process \( \{S_k\}_{k=0}^\infty \) results in a final structure \( S_m \) where every point in \( P \) 
lies on \( S_m \), and \( S_m \) represents the minimal surface enclosing \( P \) under the defined metric and further
shows that the process will eventually terminate at a state where the surface \( S \) is minimal and includes every point in \( P \).
 
\end{proof}

\section{Methodology}

The initial step in constructing a surface encompassing an $n$-point 3D point cloud involves the computation of the Convex Hull (CH). 
This entails identifying the vertex set $P^{CH} = \{p_1, p_2, \ldots, p_m\}$ of the CH and defining the corresponding facets 
to form the CH's closed surface topology. Let $\mathcal{F} = \{F_1, F_2, \ldots, F_l\}$ represent the set of facets of the CH, 
where each facet $F_j$ is a planar subset bounded by a subset of vertices in $P^{CH}$.

For the points $P_i$ not in $P^{CH}$, we define their Euclidean distance to the nearest facet $F \in \mathcal{F}$ as 
\[
d(P_i, F) = \min_{F \in \mathcal{F}} \left( \sqrt{\sum_{k=1}^{3} (x_{i,k} - x_{F,k})^2} \right),
\]
where $(x_{i,1}, x_{i,2}, x_{i,3})$ and $(x_{F,1}, x_{F,2}, x_{F,3})$ are the coordinates of point $P_i$ and the centroid of 
facet $F$, respectively.

The points $P_i$ are prioritized based on $d(P_i, F)$ in ascending order. For each selected point $P_i$, the nearest 
facet $F_k \in \mathcal{F}$ is identified and replaced with a set of three new facets $\{F_{k1}, F_{k2}, F_{k3}\}$, thereby 
forming a tetrahedron with $P_i$ as a common vertex. This modification expands the surface to encapsulate $P_i$, and $\mathcal{F}$ 
is updated to include these new facets.

To maintain the surface's integrity, points sharing a facet with a previously processed point are omitted in subsequent iterations 
to avoid the incorrect generation of new facets. This iterative procedure is repeated until all points in the point cloud are 
encompassed by the evolving surface.

Cross-intersections between facets are inherently avoided in the proposed method by leveraging geometric and topological constraints. 
Specifically, new facets are constructed by integrating a non-planar point cloud into the existing surface structure while 
adhering to minimal Hausdorff distance principles. The process can be mathematically described as follows:

\begin{enumerate}
    \item \textbf{Point Cloud Integration and Facet Formation:} \\
    Let \( P_{np} \) be a non-planar point cloud and \( F \) be the set of existing facets. For each point \( p \in P_{np} \), 
	we identify the nearest facet \( F_{\text{near}} \in F \) based on the Hausdorff distance, defined as:
    \[ d_H(p, F_{\text{near}}) = \min_{q \in F_{\text{near}}} \| p - q \|, \]
    where \( \| p - q \| \) denotes the Euclidean distance between points \( p \) and \( q \).

    \item \textbf{Minimizing Cross-Intersection:} \\
    To avoid cross-intersections, a new facet \( F_{\text{new}} \) is formed by extending \( F_{\text{near}} \) to include \( p \), 
	ensuring that \( F_{\text{new}} \) is the closest possible planar structure to \( p \) without intersecting other facets. 
	This can be expressed as:
    \[ F_{\text{new}} = \text{extend}(F_{\text{near}}, p) \]
    subject to:
    \[ \forall F_i \in F, F_i \cap F_{\text{new}} \subseteq \partial F_i \cup \partial F_{\text{new}}, \]
    where \( \partial F \) denotes the boundary of facet \( F \), and the intersection is restricted to the boundaries of the facets.

    \item \textbf{Geometric Optimization:} \\
    The extension of \( F_{\text{near}} \) to form \( F_{\text{new}} \) is optimized to ensure minimal deviation from the planar 
	structure and adherence to the surface integrity. This can be formulated as a constrained optimization problem:
    \[ \min_{F_{\text{new}}} \int_{F_{\text{new}}} d(x, P_{np}) \, dx, \]
    subject to:
    \[ F_{\text{new}} \cap F_i \subseteq \partial F_i \cup \partial F_{\text{new}}, \, \forall F_i \in F, \]
    where \( d(x, P_{np}) \) measures the point-wise distance from \( x \in F_{\text{new}} \) to the nearest point in \( P_{np} \).

    \item \textbf{Topological Consistency:} \\
    In addition to geometric considerations, topological consistency is maintained by ensuring that the addition 
	of \( F_{\text{new}} \) does not create non-manifold edges or vertices, preserving the surface topology.
\end{enumerate}

By integrating these principles, the method systematically avoids the creation of cross-intersections between facets. The approach 
not only maintains geometric fidelity to the original point cloud but also ensures topological soundness of the evolving surface.

Given the algorithmic complexity and computational intensity of this methodology, optimization and efficiency are crucial, 
especially for large-scale point clouds. Selecting appropriate data structures for storing points and facets, along with 
efficient algorithms for distance computation and facet replacement, are essential for the performance and accuracy of the 
surface envelopment process.

\subsection{Complexity Analysis}

The above algorithm's complexity can be analyzed by examining its major components: computation of the Convex Hull (CH), distance 
calculations, prioritization and processing of points, and updating the surface. 

\begin{enumerate}
    \item \textbf{Computation of the Convex Hull (CH):} The complexity for computing the CH in a 3D space is $O(n \log n)$, 
	where $n$ is the number of points in the 3D point cloud.

    \item \textbf{Distance Calculations:} For each of the $n - m$ points not in $P^{CH}$ (where $m$ is the number of 
	points in $P^{CH}$), the algorithm computes the distance to the nearest facet. Assuming the number of facets in 
	the CH is $O(m)$, the distance calculation for each point would be $O(m)$, leading to a total complexity of $O((n-m)m)$ 
	for this step.

    \item \textbf{Prioritization and Processing of Points:} The points are sorted based on their distance from the CH, which 
	has a complexity of $O((n-m) \log (n-m))$. Each point is then processed to update the facets. If we assume the average 
	number of operations to update the facets per point is constant, the complexity for processing all points would be $O(n-m)$.

    \item \textbf{Updating the Surface:} The update process involves replacing a facet with three new facets and updating the 
	data structure. Assuming constant time for each update, the complexity is again proportional to the number of points 
	being processed, i.e., $O(n-m)$.
\end{enumerate}

Combining these components, the overall complexity of the algorithm is dominated by the distance calculations and the sorting of 
points. Therefore, the total complexity is approximately:
\[
O(n \log n) + O((n-m)m) + O((n-m) \log (n-m)) + O(n-m)
\]

In the worst-case scenario (where $m$ is close to $n$), the complexity can approach $O(n^2)$. However, in practical scenarios 
where $m$ is much smaller than $n$, the $O((n-m)m)$ term may not dominate. In such cases, the complexity is more influenced by 
the $O(n \log n)$ term from the CH computation and the $O((n-m) \log (n-m))$ term from sorting the points.

This analysis, and the actual complexity can vary depending on the specifics of the implementation, 
such as the data structures used and the efficiency of the facet update process.

\section{Results}
We present the evolution of a surface, starting from an initial spatial point distribution cloud, to a final form 
toward a closed surface that encapsulates all points of the cloud in 3D environment. 
The process of contracting an initial Convex Hull surface (see figure \ref{Clouds and CH figure}) to its final all-points 
surface is demonstrated in the following. In the graphical representations, red points are used to denote those that are 
positioned on the surface.

\begin{figure}[H]
\begin{center}
\includegraphics[width=7cm]{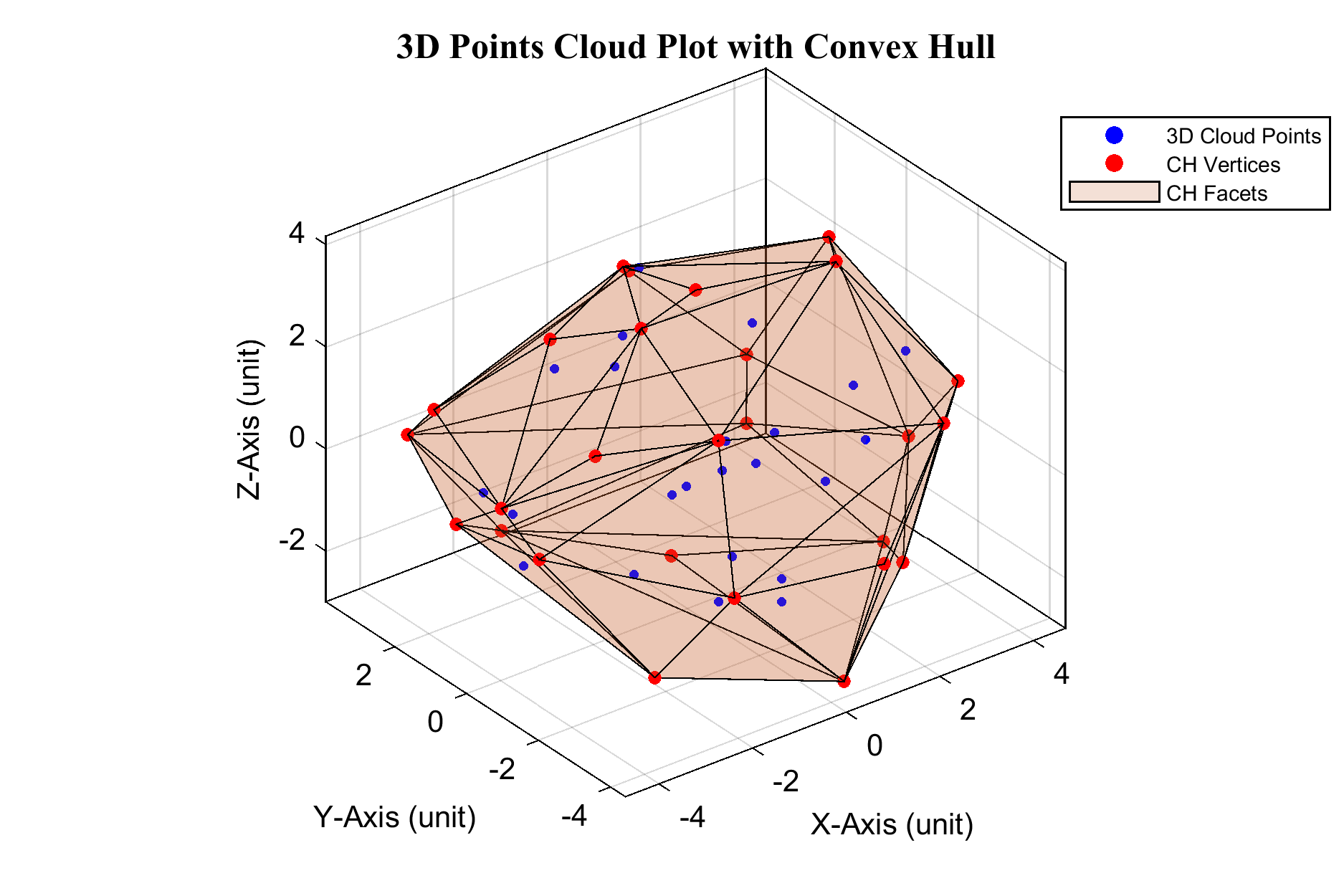}
\includegraphics[width=7cm]{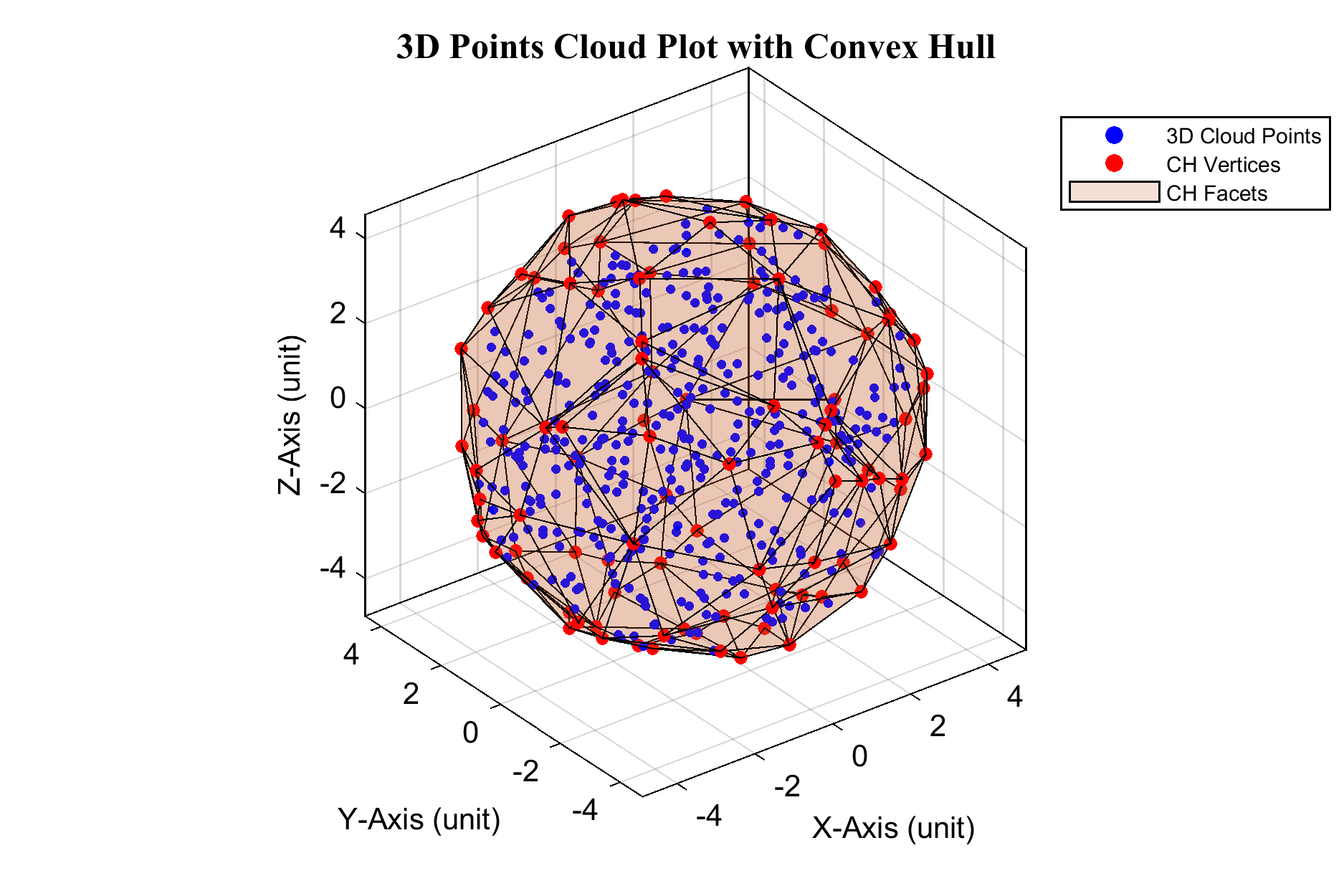}
\caption{The 3D point distributions of a 50-point (left) and a 500-point clouds sphered shape (right). Red points refer to 
on-surface points.}
\label{Clouds and CH figure}
\end{center}
\end{figure}

The process of replacing a facet and an external corresponding point with three new facets is schematically shown 
in Figure \ref{Facets figure}. Here the facet $F_{ABC} \in \mathcal{F}$ is replaced by three new facets 
$\{F_{ABP}, F_{BCP} anf F_{CAP}\}$.

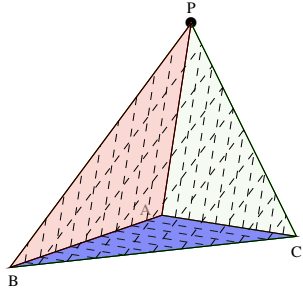
\begin{figure}[H]
\centering
\tdplotsetmaincoords{60}{120} 
\begin{tikzpicture}[scale=4,tdplot_main_coords]
    \coordinate (A) at (0,0,-0.1);
    \coordinate (B) at (1,0,0.2);
    \coordinate (C) at (0.5,0.8,0.3);

    \coordinate (P) at (0.5,0.4,1);

    \draw[fill=blue!50] (A) -- (B) -- (C) -- cycle;

    \foreach \i in {0.1,0.2,...,0.9}{
        \draw[black, dashed, very thin] ($(A)!\i!(B)$) -- ($(C)!\i!(B)$);
    }
	
    \node at (P) {\textbullet};
    \node[above, font=\tiny] at (P) {P};

    \node[left, font=\tiny, yshift=0.4ex] at (A) {A};
    \node[right, font=\tiny, xshift=-0.9ex, yshift=-1.0ex] at (B) {B}; 
    \node[below, font=\tiny] at (C) {C};

    \draw[fill=brown!30, draw=brown, fill opacity=0.1] (A) -- (C) -- (P) -- cycle;
    \draw[fill=green!30, draw=green, fill opacity=0.1] (B) -- (C) -- (P) -- cycle;
    \draw[fill=red!30, draw=red, fill opacity=0.5] (A) -- (B) -- (P) -- cycle;

    \foreach \i in {0.1,0.2,...,0.9}{
        \draw[black, dashed, very thin] ($(A)!\i!(C)$) -- ($(P)!\i!(C)$);
    }
    \foreach \i in {0.1,0.2,...,0.9}{
        \draw[black, dashed, very thin] ($(B)!\i!(C)$) -- ($(P)!\i!(C)$);
    }
    \foreach \i in {0.1,0.2,...,0.9}{
        \draw[black, dashed, very thin] ($(A)!\i!(B)$) -- ($(P)!\i!(B)$);
    }
	
    \draw (A) -- (B);
    \draw (B) -- (C);
    \draw (C) -- (A);
    \draw (A) -- (P);
    \draw (B) -- (P);
    \draw (C) -- (P);
	
\end{tikzpicture}
\caption{Replacement of a 3-point facet $F_{ABC}$ and an external point $P$ into the tetrahedron
represented by $\{F_{ABP}$, $F_{BCP}$ and $F_{CAP}\}$}.
\label{Facets figure}
\end{figure}

The final surfaces, resulting from the contraction process detailed earlier, are depicted in Figure \ref{Final Surfaces figure}. 
These surfaces, while intricate and challenging to visualize, successfully encapsulate all points without any intersecting facets.

\begin{figure}[H]
\begin{center}
\includegraphics[width=7cm]{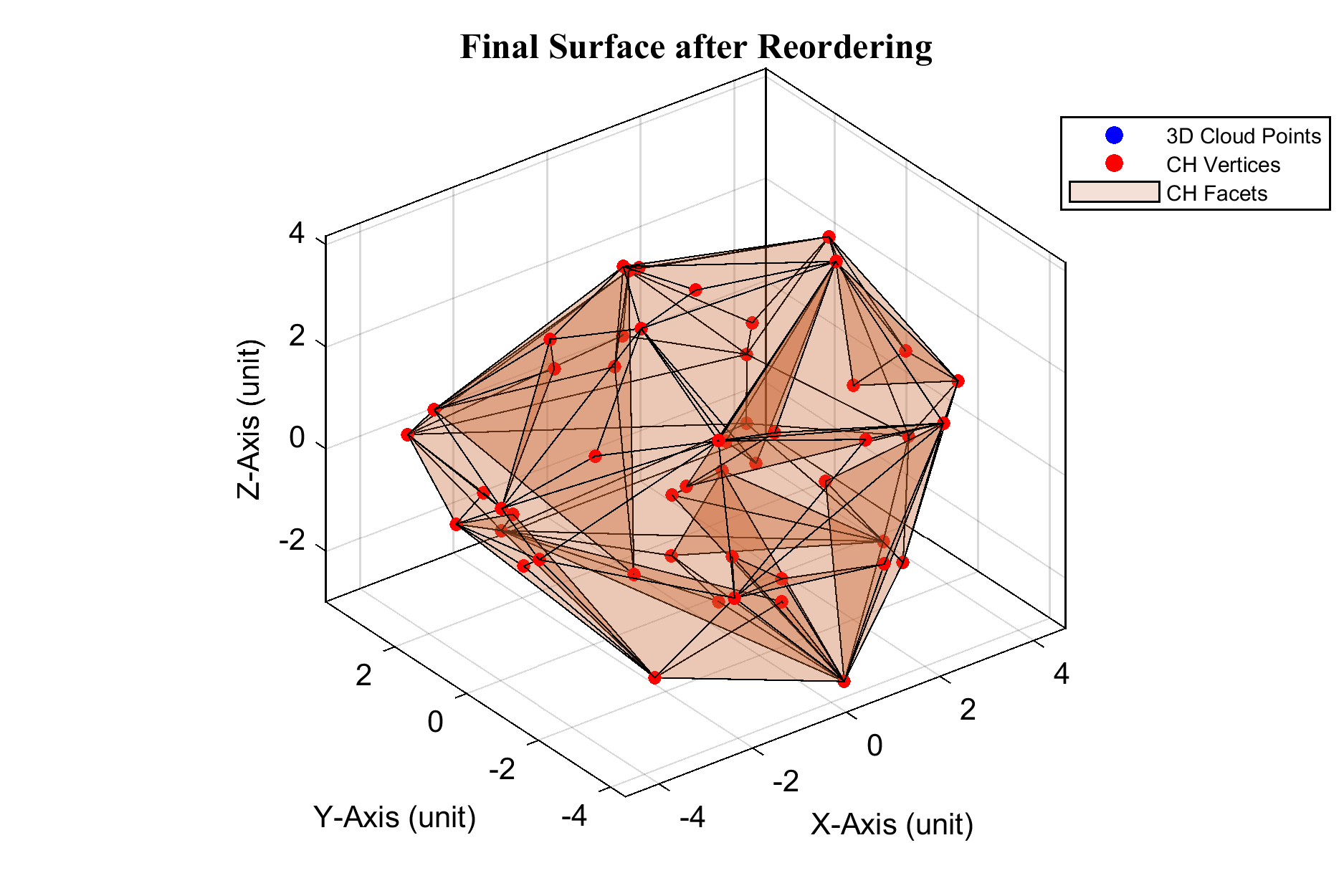}
\includegraphics[width=7cm]{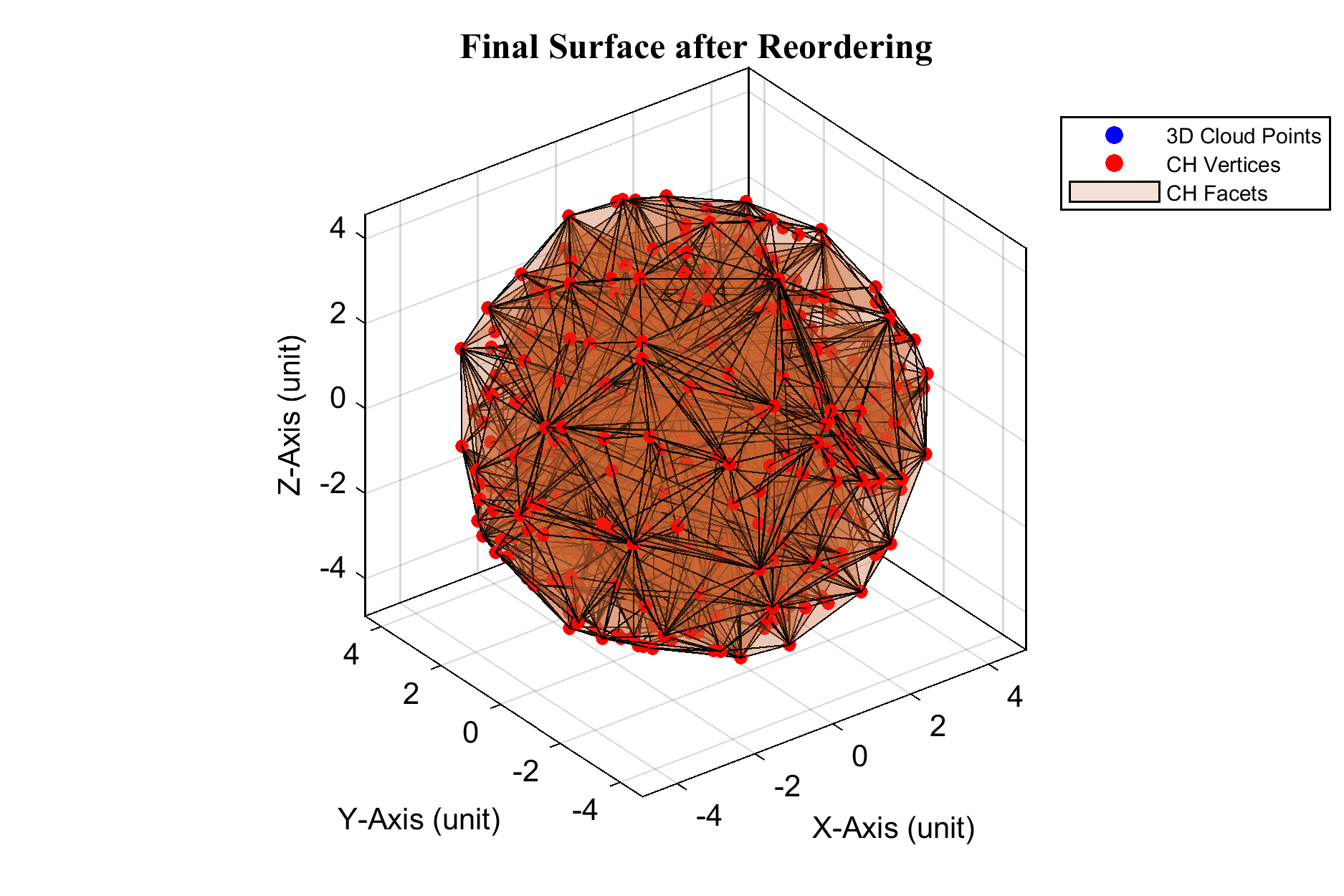}
\caption{The final surfaces encapsulating all 50 points (left) and 500 points (right). Red points refer to on-surface points.}
\label{Final Surfaces figure}
\end{center}
\end{figure}

Several key observations are to be considered:

\begin{itemize}
    \item \textbf{Comprehensive Inclusion of Points:} In all simulated scenarios, we observed that the entire set of 
	points \( P_i \in P \) were consistently and accurately enclosed within the final surface. This indicates the robustness 
	of our method in encompassing every element of the point cloud, regardless of its spatial distribution, in accordace with 
	Theorem \ref{thm:Contraction to Point Cloud Envelope}.

    \item \textbf{Absence of Cross-Intersection:} Throughout the simulations, no intersecting facets were identified, affirming 
	both the closure and integrity of the final surfaces. This outcome highlights the effectiveness of the geometric algorithm in 
	maintaining topological soundness, ensuring that the resulting surface is both continuous and non-self-intersecting.

    \item \textbf{Process Convergence:} Empirical analysis revealed that for a typical three-dimensional point cloud consisting 
	of \( n \) points, approximately \( \frac{n}{100} \) iterations are required to ensure complete enclosure of all points within 
	the final surface. This rate of convergence demonstrates the efficiency of our method, particularly in its ability to rapidly 
	adapt to the spatial characteristics of the point cloud.
	
    \item \textbf{Computational Complexity:} A detailed preliminary analysis was conducted to estimate the computational complexity 
	of the algorithm. In the worst-case scenarios, characterized by highly dense and irregular point clouds, the complexity is 
	estimated to be \( O(n^2) \), where \( n \) represents the number of points in the cloud. This quadratic complexity arises 
	from pairwise comparisons between points during the facet formation process. On the other hand, in more practical and common 
	cases, especially when the point cloud exhibits certain spatial regularities or sparseness, the complexity is notably reduced. 
	It is approximated to be \( O((n-m) \log (n-m)) \), where \( m \) signifies the number of points already processed or enclosed 
	by the surface. This logarithmic factor indicates efficiency gains through optimizations such as spatial indexing or efficient 
	nearest-neighbor search algorithms.
	
\end{itemize}

\section*{Conclusion}

In this study, we have investigated the transformation of a convex hull (\( S^{ch} \)) derived from a d-dimensional point cloud 
into a concave surface (\( S^{cc} \)). Our primary objective was to encapsulate all points of the point cloud within a closed 
concave surface, while retaining the surface's non-intersecting property. This transformation is crucial for a more accurate 
representation of the underlying geometry of the point cloud, especially when dealing with complex structures exhibiting concave 
features.

Our results demonstrate that the iterative process of transforming \( S^{ch} \) into \( S^{cc} \) is effective in encompassing 
all points in the point cloud. The iterative procedure, consisting of facet replacement and expansion, ensured that each point 
of the point cloud was eventually included on the surface. We observed that in typical three-dimensional point clouds, a 
relatively small number of iterations (approximately \( \frac{n}{100} \)) were sufficient to achieve complete enclosure of all points.
 
As part of the study, we have formulated and rigorously proven the \textit{Point Cloud Contraction Theorem}. This theorem offers a 
theoretical foundation for the transformation process of the convex hull into a minimal-area, closed concave surface, while ensuring 
the inclusion of all points from the initial spatial cloud.
The theorem articulates and mathematically validates the premise that the iterative approach will unfailingly yield a surface 
that comprehensively encompasses all points within the point cloud. By providing this proof, we reinforce the reliability 
of the transformation process. 

Through empirical analysis, it was observed that the final surface \( S^{cc} \) successfully represented the intricate geometry of 
the point cloud without any facet intersections. This aspect underlines the robustness of the method in maintaining the topological 
integrity of the surface, which is particularly pertinent in applications like 3D modeling and spatial analysis.

The computational complexity of the approach was found to vary based on the point cloud's characteristics. In the most complex 
cases, the complexity reached \( O(n^2) \), while in more typical scenarios, it was significantly lower, approximating 
\( O((n-m) \log (n-m)) \). This variation highlights the need for efficient computational strategies, especially in handling larger 
and more complex point clouds.

We note that the theoretical approach outlined in this study is versatile and can be adapted to higher-dimensional spaces as well 
as to planar scenarios.

In summary, our approach presents a methodical way to transition from a convex hull to a concave surface, capturing the geometric 
intricacies of d-dimensional point clouds. While our method is efficient and effective in certain scenarios, its performance is 
subject to the distribution and density of the point cloud. The study contributes to the broader understanding of geometric 
transformations in computational geometry and offers a basis for future research in surface reconstruction and analysis, 
especially in fields requiring accurate geometric representations of complex structures.

\section*{Declarations}
All data-related information and coding scripts discussed in the results section are available from the 
corresponding author upon request.

\renewcommand{\bibname}{References}

\begin{thebibliography}{99}

\bibitem[Loffler and van Kreveld(2010)]{Loffler:2010}
M. Löffler and M. van Kreveld.
\newblock Largest and Smallest Convex Hulls for Imprecise Points.
\newblock \emph{Algorithmica}, 56(2):235--269, 2010.
\newblock \url{https://doi.org/10.1007/s00453-008-9174-2}
\newblock DOI: 10.1007/s00453-008-9174-2.

\bibitem[Moriya(2023)]{Moriya:2023smallest}
Netzer Moriya.
\newblock Smallest Enclosing Sphere in 3D -- Particle Swarm Optimization Approach.
\newblock arXiv:2311.03843 [math.OC], 2023.
\newblock \url{https://doi.org/10.48550/arXiv.2311.03843}

\bibitem[Karzanov(2005)]{Karzanov:2005}
Alexander V. Karzanov.
\newblock Concave cocirculations in a triangular grid.
\newblock \emph{Linear Algebra and its Applications}, 400:67--89, 2005.
\newblock \url{https://doi.org/10.1016/j.laa.2004.11.003}

\bibitem[Asaeedi et al.(2013)]{Asaeedi:2013}
Saeed Asaeedi, Farzad Didehvar, and Ali Mohades.
\newblock Alpha Convex Hull, a Generalization of Convex Hull.
\newblock \emph{ArXiv}, abs/1309.7829, 2013.
\newblock Available at \url{https://api.semanticscholar.org/CorpusID:36495981}

\bibitem[Park and Oh(2012)]{Park:2012}
Jin-Seo Park and Se-Jong Oh.
\newblock A New Concave Hull Algorithm and Concaveness Measure for n-dimensional Datasets.
\newblock \emph{J. Inf. Sci. Eng.}, 29:379--392, 2012.
\newblock Available at \url{https://api.semanticscholar.org/CorpusID:10272845}

\bibitem{Moriya:2024TheLargest}
Netzer Moriya.
\newblock The Largest Empty Sphere Problem in 3D Hollowed Point Clouds.
\newblock \emph{arXiv:2401.07593 [math.OC]}, 2024.
\newblock DOI: \url{https://doi.org/10.48550/arXiv.2401.07593}.

\bibitem[Pramila and Virtanen(1986)]{Pramila:1986}
Antti Pramila and Simo Virtanen.
\newblock Surfaces of minimum area by FEM.
\newblock \emph{International Journal for Numerical Methods in Engineering}, 23:1669--1677, 1986.
\newblock Available at \url{https://api.semanticscholar.org/CorpusID:122483925}

\bibitem{Yang:2010}
Yi Yang, Mengyin Fu, Wei Wang, Xin-miao Yang, and Hao Zhu.
\newblock 3D laser point cloud-based navigation in complex environment.
\newblock In \emph{Proceedings of the 29th Chinese Control Conference}, pages 3798--3803, 2010.
\newblock Available at \url{https://api.semanticscholar.org/CorpusID:13279792}.

\bibitem[Alliez et al.(2007)]{Alliez:2007}
Pierre Alliez, David Cohen-Steiner, Yiying Tong, and Mathieu Desbrun.
\newblock Voronoi-based variational reconstruction of unoriented point sets.
\newblock In \emph{Symposium on Geometry processing}, vol. 7, pp. 39--48, 2007.

\bibitem{rudin:1976}
Walter Rudin.
\newblock \textit{Principles of Mathematical Analysis}.
\newblock McGraw-Hill, 1976.
\newblock ISBN 9780070542358.

\bibitem{Ugwunnadi:2019}
G. C. Ugwunnadi, C. Izuchukwu, and O. T. Mewomo.
\newblock Strong convergence theorem for monotone inclusion problem in CAT(0) spaces.
\newblock \emph{Afrika Matematika}, 30(1):151--169, 2019.
\newblock \url{https://doi.org/10.1007/s13370-018-0633-x}
\newblock DOI: 10.1007/s13370-018-0633-x.





\end{thebibliography}

\end{document}